\documentclass[11pt]{amsart}
\usepackage{a4wide}
\usepackage[dvips]{epsfig}
\usepackage{amscd}
\usepackage{amssymb}
\usepackage{amsthm}
\usepackage{amsmath}
\usepackage{latexsym}
\usepackage{mathtools}

\theoremstyle{plain}
\newtheorem{thm}{Theorem}[section]
\newtheorem{defn}{Definition}[section]
\theoremstyle{plain}

\theoremstyle{definition}

\newtheorem{asp}{Assumption}

{%
\setcounter{enumi}{0}

\begin{enumerate}}%
{\end{enumerate} }
{%
\setcounter{enumi}{0}

\begin{enumerate}}%
{\end{enumerate} }

\numberwithin{equation}{section}
 \allowdisplaybreaks

\title[Sensitivities for path-dependent options]
 {Sensitivity Analysis of Path-Dependent Options in an Incomplete Market with Pathwise Functional It$\hat{\text{O}}$ Calculus}

\date{\today}

\begin{document}
\author{ Siboniso Confrence Nkosi }

\address{Department of Mathematics and Applied Mathematics, University of Limpopo, Private bag X1106, Sovenga, 0727, South Africa}

\email{siboniso.nkosi@ul.ac.za}

\author{ Farai Julius Mhlanga }

\address{Department of Mathematics and Applied Mathematics,
University of Limpopo, Private bag X1106, Sovenga, 0727, South Africa}

\email{farai.mhlanga@ul.ac.za}

\keywords{path-dependent options, pathwise functional It$\hat{\text{o}}$ calculus, Price sensitivities }





\begin{abstract}
Functional It{\^o} calculus is based on an extension of the classical It{\^o} calculus to functionals depending on the entire
past evolution of the underlying paths and not only on its current value. The calculus builds on F{\"o}llmer's deterministic proof
of the It{\^o} formula, see \cite{follmer1981calcul}, and a notion of pathwise functional derivatives introduced
by \cite{dupire2009functional}. There are no smoothness assumptions required on the functionals, however, they are required to
possess certain directional derivatives which may be computed pathwise, see \cite{cont2013functional,cont2012functional,schied2016associativity}. 
Using functional It{\^o} calculus and the notion of quadratic variation, we derive the functional It{\^o} formula along with the Feynman-Kac formula for functional processes. Furthermore, we express the Greeks for path-dependent options as expectations, which can be efficiently computed numerically using Monte Carlo simulations. We illustrate these results by applying the formulae to digital options within the Black-Scholes model framework.
\end{abstract}

\maketitle
\section{Introduction}
Many problems in stochastic analysis and its applications to mathematical finance involve the study of path-dependent functionals of stochastic processes. In financial markets, portfolio strategies are often constructed based on historical data rather than solely on the current market price or capitalisation. This naturally raises the question of whether a mathematical framework can be developed to analyse portfolios that depend on the entire past evolution of the market portfolio and other influencing factors.
\\
In his seminal work, F{\"o}llmer \cite{follmer1981calcul} introduced a pathwise proof of It{\^o} formula using the concept of quadratic variation along a sequence of partitions. This result has gained significant attention in the pathwise approach to stochastic analysis and it is flexible to be extended to path-dependent functionals. F{\"o}llmer's ideas laid the foundation for subsequent developments by Dupire \cite{dupire2009functional} and Cont and Fourni{\'e} \cite{cont2013functional, cont2010functional}, who introduced a novel stochastic calculus known as functional It{\^o} calculus. This framework extends the classical It{\^o} formula to functionals that depend on the entire past evolution of the underlying process rather than just its current value. While Dupire \cite{dupire2009functional} employed probabilistic arguments and It{\^o} calculus for functional representation, alternative approaches allow for analytical derivations without relying on probability. This has led to the development of pathwise functional It{\^o} calculus for non-anticipative functionals, which identifies the class of c{\'a}dl{\'a}g paths to which the calculus is applicable.
\\
Predicting the price of an option or a portfolio comprising multiple options remains a challenging task due to the market fluctuations that do not always correspond directly to changes in the underlying asset's price. Various factors influence option pricing, making it essential to understand and quantify their effects. The Greeks serve as crucial tools in this context, measuring the sensitivity of option prices by quantifying the impact of key parameters. In risk management, Greeks play a fundamental role in assessing a portfolio's sensitivity to small changes in underlying parameters, making their computation essential. For practical applications, deriving explicit formulas for Greeks is crucial for designing an efficient Monte Carlo evaluation. Various representations of Greeks have been evaluated using Malliavin calculus, including those for jump diffusion processes \cite{davis2006, mhlanga2015}, quanto options \cite{mhlanga2022}, spread options \cite{deel2010, mhlanga2023}, the infinite-dimensional Heston model \cite{dinunno2021}, and hybrid stochastic volatility models \cite{khatib2023, yilmaz2018}, to name a few.
\\
The paper by Jazaerli and Saporito \cite{jazaerli2017functional} presented a novel approach using the Lie bracket of space and time derivatives, classifying functionals by their degree of path-dependence. The Lie bracket here behaves as a measure of path-dependence, exposing the authors to an alternative approach for the computation of Greeks. They explored integration by parts techniques for numerical computation of Greeks, emphasizing cases with locally weakly and fully path-dependent functionals. Their work covers models with path-dependent volatility, which are not adequately addressed in the Malliavin calculus setup. The effectiveness of their framework is demonstrated using examples like Asian options and contracts with quadratic variation.
\\
Yu and Wang \cite{yu2023quasi} developed a generalised integration by parts formula in multi-dimensional Malliavin calculus and used it to derive Greeks for complex Asian options. They introduced a method combining Malliavin calculus with the Quasi-Monte Carlo (QMC) method, which enhances efficiency in computing option Greeks. Their approach smooths Malliavin Greeks estimates using conditional expectations, improving numerical stability and reducing variance. They computed Greeks for both simple and complex Asian options in a multi-asset setting, employing gradient principal component analysis and other variance reduction techniques to optimise QMC simulations. Their numerical experiments demonstrated significant improvements, particularly for Asian options with discontinuous payoffs.
\\
Yang \textit{et al.} \cite{yang2021analysis} derived closed-form matrix expressions for the Laplace transforms of price and Greeks of Asian options using pathwise differentiation and Malliavin calculus. They established exact second-order convergence rates of continuous-time Markov Chain (CTMC) methods when applied to the price and Greeks of Asian options, which was missing in literature. Their work introduced new error analysis techniques for CTMC methods applied to path-dependent derivatives, enhancing accuracy and efficiency. Furthermore, they generalised the CTMC framework to occupation-time derivatives, which are frequently used in financial markets such as equity and forex options. Numerical experiments validated the efficiency of the CTMC approach for path-dependent option pricing.
\\
Hudde and R{\"u}schendorf \cite{hudde2023european} computed closed-form expressions for European and Asian Greeks for general $L^2$-payoff functions. They employed Hilbert space-valued Malliavin calculus to derive stochastic weight formulas for European and Asian options. Their findings demonstrated that finite difference methods perform better for continuous payoff functions, whereas Malliavin Monte Carlo methods exhibit superior convergence for discontinuous payoff functions, such as digital options.
\\
The Malliavin calculus, a differential framework on Wiener space, has been instrumental in studying Brownian functionals with path-dependent features. However, in the Malliavin setup, accounting for more general dynamics of path-dependent volatility remains a challenge. Functional It{\^o} calculus has proven to be an effective tool in addressing these complexities, particularly in computing the Greeks. It{\^o} stochastic calculus has long been recognised as a powerful tool for analysing phenomena characterised by randomness and irregular time evolution \cite{alos2001stochastic,applebaum2009levy}. At its core, the It{\^o} formula enables the representation of stochastic integrals, making it key to modern stochastic analysis. Due to the broad range of applications spanning from physics to mathematical finance, significant efforts have been made to extend stochastic calculus to path-dependent functionals.
\\
The Malliavin calculus naturally leads to functional representations involving anticipative processes \cite{bally2003elementary,watanabe1984lectures}. However, its derivatives introduce perturbations that affect both past and future values of the process, making it less interpretable in contexts where the quantities are expected to be non-anticipative. This limitation underscores the importance of functional It{\^o} calculus in developing robust methodologies for analysing path-dependent stochastic systems in mathematical finance.
\\
The contribution of this paper is twofold. First, we derive a Feynmann-Kac type formula for functionals, extending the result in \cite{dupire2009functional}. Second, we use functional It{\^o}  calculus to compute the Greeks for path-dependent options, specifically Delta, Gamma, and Vega. Delta is defined as the derivative of the price of the path-dependent option with respect to the initial price. Gamma is defined as the second derivative of the price of the path-dependent option with respect to the initial price. Vega is defined as the derivative of the price of the path-dependent option with respect to the diffusion coefficient. A Lie bracket of time and space functional derivatives plays a fundamental role in deriving Greeks for path-dependent options. In addition, we provide the Black-Scholes versions of these Greeks as examples.
\\
The remainder of the paper is organised as follows. Section 2 introduces the fundamental properties of pathwise functional It{\^o} calculus and defines the Lie bracket. This section also classifies functionals as either locally weakly path-dependent or strongly path-dependent. Section 3 presents the asset price model driven by a Brownian motion and a jump process. It also provides the risk-neutral pricing formula for a path-dependent derivative with maturity $T$. Section 4 explores Functional It{\^o} Calculus, utilising quadratic variation to describe how functionals accumulate stochastic fluctuations over time. The section extends the classical It{\^o} formula to path-dependent functionals and derives the pricing partial differential equation (PDE). Section 5 focuses on the computation of Greeks, particularly the first-order Greeks (Delta and Vega) and the second-order Greek (Gamma). Section 6 examines a jump-diffusion model with zero interest rates, applying the previously derived formulae to compute Delta, Gamma and Vega for digital options. Section 7 concludes the paper and discusses potential directions for future research.

\section{A Primer on pathwise functional It{\^o} calculus}
In this section we recall some of the basic properties of pathwise functional It{\^o} calculus as highlighted in \cite{Blan2020}. We refer to Cont \cite{cont2012functional} for a detailed exposition on pathwise functional It{\^o} calculus.\\
Let $\mathcal D_T=D([0,T],\mathbb R^d)$ be the space of c{\'a}dl{\'a}g paths on $[0,T]$, where $T<\infty$. The value of the path $\omega\in \mathcal D_T)$  at a fixed time $t\in[0,T]$ is denoted by $\omega(t)$.
\begin{defn}
For a fixed time $t\in[0,T]$,  a stopped path $\omega_t$ is obtained by fixing the value of the path $\omega$ on $(t,T]$ at $\omega(t)$, that is,
\begin{displaymath}
\omega_t(s)=\omega(t\wedge s)=\left\{\begin{array}{ll} & \omega(s)~~\text{for}~~s\in[0,t],\\
& \omega(t)~~\text{for}~~s\in(t,T].
\end{array}\right.
\end{displaymath}
\end{defn}
Let ${\mathcal B}(A)$ denote the Borel $\sigma$-algebra on the set $A$. Let $\mathbb R_0^d:=\mathbb R^d -\{0\}$ and ${\mathcal M}_T:={\mathcal M}([0,T]\times\mathbb R_0^d)$ be the space of integer-valued Radon measures on $[0,T]\times\mathbb R_0^d$.
\begin{defn}
For a measure $j:\mathcal B([0,T]\times\mathbb R_0^d)\rightarrow\mathbb N\cup\{+\infty\},$ we say that
\[j\in{\mathcal M}_T \Leftrightarrow j(\cdot)=\sum_{i=1}^{\infty}\delta_{(t_i,z_i)}(\cdot)~~\text{and is finite on compacts} \]
with $(t_i)_{i\in\mathbb N}\in[0,T]^{\mathbb N}$ not necessarily distinct or ordered, and $(z_i)_{i\in\mathbb N}\in(\mathbb R_0^d)^{\mathbb N}$.
 \end{defn}
 The space $\mathcal M_T$ is equipped with $\sigma$-algebra $\mathcal F$ such that the mapping $j\rightarrow j(A)$ is measurable for all $A\in\mathcal B([0,T]\times\mathbb R_0^d)$, the Borel $\sigma$-algebra on $[0,T]\times\mathbb R_0^d$.
 \begin{defn}
For any $(t,j)\in[0,T]\times \mathcal M_T$ we denote $j_t$ the restriction of $j$ to $[0,t]$:
\[ \forall~A\in\mathcal B(\mathbb R_0^d),j_t([0,t]\times A)=j([0,t]\times A)\]
and $j_{t-}$ the restriction of $j$ to $[0,t)$:
\[ \forall~A\in\mathcal B(\mathbb R_0^d),j_{t-}([0,t]\times A)=j([0,t)\times A).\]
\end{defn}
\begin{defn}
We identify the space of processes
 \[\Omega:=\mathcal D_T \times \mathcal M_T\]
   equipped with the $\sigma$-algebra ${\mathcal F}$ defined as the $\sigma$-algebra on the product space such that for every $j\in \mathcal M_T$ and $A\in\mathcal B(\mathbb R_0^d)$, $j(A)$ is measurable.
   \end{defn}
We note that if $\mathcal D_T$ is equipped with the Skorohod topology and $\mathcal M_T$ with topology of weak convergence, both are separable, and there is no difference between the product Borel $\sigma$-algebra and the product of the Borel $\sigma$-algebras defined respectively on $\mathcal D_T$ and $\mathcal M_T$.\\
The canonical process $(\omega,J)$ is defined as follows: for any $(\omega,j)\in\Omega$,
\[(\omega,J)=(\omega_t(\cdot),j_t(\cdot)), \]
that is, the process consists of the trajectory and the measure $j$ both stopped at time $t$.
\begin{defn}
For a given $(j_t)_{t\in[0,T]}$, we define the filtration generated by the canonical process $(\omega,J):$ $\mathbb F:=(\mathcal F_t)_{t\in[0,T]}$ on $(\Omega,\mathcal F)$ as the increasing sequence of $\sigma$-algebras
\[  \mathcal F_t=\sigma(\omega_s(\cdot),J_s(\cdot), ~s\in[0,t]). \]
\end{defn}

Now, we define non-anticipative functionals on the space \cite{Blan2020}.
\begin{defn}
 A non-anticipative functional is a map
 \[F:[0,T]\times\Omega\rightarrow\mathbb R\]
 such that
 \begin{enumerate}
   \item $F(t,\omega,j)=F(t,\omega_t,j_t),$
   \item $F$ is measurable with respect to the product $\sigma$-algebra $\mathcal B([0,T]\times\mathcal F)$,
   \item For all $t\in[0,T]$, $F(t,\cdot)$ is $\mathcal F_t$-measurable.
 \end{enumerate}
 \end{defn}
 \begin{defn}
 A predictable functional $F$ is a non-anticipative functional such that
 \[ F(t,z,\omega,j)=F(t,z,\omega_t,j_t)=f(t,z,\omega_t,j_{t-}),\]
 and so $F(t,z,\cdot,\cdot)$ is $\mathcal F_t$-measurable.
 \end{defn}
 \begin{defn}
 A non-anticipative functional field $\Psi$ is a map
 \[\Psi:[0,T]\times\mathbb R_0^d\times\Omega\rightarrow\mathbb R\]
 such that
 \begin{enumerate}
   \item $\Psi(t,z,\omega,j)=\Psi(t,z,\omega,j_t)$,
   \item $\Psi$ is measurable with resoect to the product $\sigma$-algebra $\mathcal B([0,T]\times\mathbb R_0^d)\times\mathcal F$,
   \item for all $(t,z)\in[0,T]\times\mathbb R_0^d$, $\Psi(t,z,\cdot)$ is $\mathcal F_t$-measurable.
 \end{enumerate}
  \end{defn}

%
\begin{defn}[Boundedness-preserving functionals]
A non-anticipative functional $F$ is boundedness-preserving if, for every compact set $K\subset\mathbb{R}^d$, there exists a constant $C(K,t_0)>0$, such that $|F(t,Y_t)|<C(K,t_0)$, for every path $Y_t$ satisfying $Y_t([0,t])=\{y\in\mathbb R: Y_t(s)=y~~\text{for some}~~s\in[0,t]\}\subset K.$
\end{defn}
We denote by $\mathbb{B}(\Omega)$ the set of boundedness-preserving non-anticipative functionals.\\

We now recall some notions of differentiability for functionals \cite{Blan2020, cont2012functional, dupire2009functional}. We consider a path $\omega\in \Omega$. For $h\geq0$, the horizontal extension of a stopped path $(t,\omega_t,j_t)$ to $[0,t+h]$ is the stopped path $(t+h,\omega_t,j_t)$.
\begin{defn}
A non-anticipative functional $F:\Omega\to\mathbb{R}$ is said to be horizontally differentiable at $(t,\omega,j)\in[0,T]\times\Omega$ if the following limit (called the horizontal derivative) exists
\begin{align}\label{hor_der}
\lim_{h\to0^+}\frac{F(t+h,\omega_t,j_t)-F(t,\omega_t,j_t)}{h}=:\mathcal{D}F(t,\omega,j).
\end{align}
 If $F$ is horizontally differentiable at each $(t,\omega,j)\in [0,T]\times\Omega$ then the map $\mathcal{D}F: (t,\omega,j)\rightarrow\mathcal {D}F(t,\omega,j)$ defines a non-anticipative functional, called the horizontal derivative of $F$.
\end{defn}

We consider $h\in\mathbb R$ and define the vertical perturbation $\omega_t^h$ of $\omega_t$ as a c{\'a}dl{\'a}g path obtain by shifting the path $\omega$ by $h$ at time $t$,
\[\omega_t^h=\omega_t(\cdot)+h\mathbf{1}_{[t,\infty]}(\cdot) \] where $\mathbf{1}_{[t,\infty]}$ is the indicator function.
\begin{defn}
A non-anticipative functional $f$ is vertically differentiable at $(t,\omega,j)\in[0,T]\times\Omega$ if, for $(e_i)_{i\in 1,...,d}$ the canonical basis of $\mathbb R^d$, and $h>0$ real, the following limit exists
\[\partial_iF(t,\omega,j)=\lim_{h\rightarrow0}\frac{F(t,\omega_t^{he_i},j_t)-F(t,\omega_t,j_t)}{h},~~i=1,...,d. \]
The vertical derivative at $(t,\omega,j)\in\Omega$ is then defined as
\[\nabla_\omega F(t,\omega,j)=\left(\partial_iF(t,\omega,j),~i=1,\ldots,d\right). \]
If $F$ is vertically differentiable at each $(t,\omega,j)$, then $\nabla_{\omega}F(t,\omega,j)$ defines a non-anticipative functional called the vertical derivative of $F$.
\end{defn}
The following definitions defines a functional derivative with respect to the jump component.
\begin{defn}
For a  non-anticipative functional $F$ and $(t,z)\in[0,T]\times\mathbb R_0^d$ we have
\[ \nabla_{t,z}F(t,\omega,j)=F(t,\omega_t,j_{t-}+\delta_{(t,z)})-F(t,\omega_t,j_{t-}).\]
Then the operator
\begin{align*}
(\nabla_JF)(t,\omega,j):&[0,T]\times\Omega\to\mathbb R\\
&(t,z,\omega,j)\mapsto\nabla_JF(t,z,\omega,j)
\end{align*}
defined by
\[ (\nabla_JF)(t,z,\omega,j)=\nabla_{t,z}F(t,\omega,j)=F(t,\omega_t,j_{t-}+\delta_{(t,z)})-F(t,\omega_t,j_{t-}) \]
maps a non-anticipative functional $F$ to a predictable functional field denoted $\nabla_JF$.
\end{defn}

We conclude this section by defining the functional Lie bracket. Unlike the usual partial derivatives in finite dimensions, the horizontal and vertical derivatives do not commute. In general
\[\mathcal{D}(\nabla_\omega F)\neq\nabla_\omega(\mathcal{D}F). \]
This arises from the fact that the elementary operators of horizontal extension and vertical perturbation defined earlier on do not commute, that is, a horizontal extension of the stopped path $(t,\omega)$ to $t+h$ followed by a vertical perturbation yields the path $\omega_t+e\mathbf{1}_{[t+h,t]}$ while a vertical perturbation at $t$ followed by a horizontal extension to $t+h$ yields
\[\omega_t+e\mathbf{1}_{[t,t]}\neq\omega_t+e\mathbf{1}_{[t+h,t].} \]
The following definition can be used to quantify the path-dependency of $F$.
\begin{defn}\label{Lie21}
The Lie bracket of the operators $\mathcal{D}$ and $\nabla_\omega$ is defined by
\[\mathfrak{L}F(t,\omega_t):=[\mathcal{D},\nabla_\omega]F(t,\omega_t)=\nabla_\omega(\mathcal{D}F)(t,\omega)-\mathcal{D}(\nabla_\omega F)(t,\omega), \]
where $F$ is such that all derivatives above exists.
\end{defn}

\begin{defn}
A function $F:\Lambda\rightarrow\mathbb R$ is called
\begin{enumerate}
  \item locally weakly path-dependent if $\mathfrak{L}F=0$, and
  \item strongly path-dependent if $\forall [s,t]\subset[0,T],~\exists u\in[s,t]$ such that $\mathfrak{L}F(Y_u)\neq0$.
\end{enumerate}
\end{defn}
\section{Mathematical setup}
Consider a $d$-dimensional Brownian motion $W$ and a Poisson random measure $N$ on $[0,\infty)\times\mathbb R$ with intensity measure $dt\times\nu(dx)$ defined on a probability space $(\Omega,{\mathcal F},P)$, where $\nu$ is a positive measure on $\mathbb R$ such that $\int_{\mathbb R}(1\wedge x^2)\nu(dx)<\infty$. $\widetilde{N}$ denotes the compensated version of $N$:
\[ \widetilde{N}(\mathrm{d}t,\mathrm{d}z)=N(\mathrm{d}t,\mathrm{d}z)-\mathrm{d}t\nu(\mathrm{d}z).\]
Let $({\mathcal F}_t)_{t\geq0}$ stand for the natural filtration of $W$ and $N$ completed with null sets.\\
We consider a stochastic process $x$ described by the following stochastic integral equation
\begin{equation}\label{a}
x_s=x_0+\int_0^sb(X_{u})\,\mathrm{d}u+\int_0^s\sigma(X_{u})\,\mathrm{d}w_u+\int_0^s\int_{\mathbb R}\gamma(z,X_{u-})\widetilde{N}(\mathrm{d}u,\mathrm{d}z)
\end{equation}
with $s\geq t$ and $X_t=Y_t$, where $b(x)$ $\sigma(x)$ and $\gamma(z,x)$ are deterministic functions. The process $(w_s)_{s\in[0,T]}$ represents a standard Brownian motion. We assume that the coefficients satisfy the following assumptions which guarantee the existence and uniqueness of a solution to the stochastic integral equation (\ref{a}):
\begin{enumerate}
  \item  For all $x\in\mathbb R^d$ and $t\in[0,T]$, there exists a constant $C_1<\infty$ such that
  \[ \mid b(x)\mid^2+\mid\sigma(x)\mid^2+\int_{\mathbb R}\mid\gamma(z,x)\mid^2\nu(\mathrm{d}z)\leq C_1(1+\mid x\mid^2), \]
  \item  For all $x,y\in\mathbb R^d$ and $t\in[0,T]$, there exists a constant $C_2$ such that
  \[ \mid b(x) - b(y)\mid^2+\mid\sigma(x) - \sigma(y)\mid^2+\int_{\mathbb R}\mid \gamma(z,x) - \gamma(z,y)\mid^2\nu(\mathrm{d}z)\leq C_2\mid x-y\mid^2.\]
\end{enumerate}
Suppose the the dynamics of a stock price $x$, under the risk-neutral measure, is defined by the path-dependent model (\ref{a}). A no-arbitrage price of a path-dependent derivative with maturity $T$ and payoff given by the functional $g:\Omega\rightarrow\mathbb R$  is given by
\[f(Y_t)=\mathbb E\left[e^{-\int_t^Tr(Z_u)du}g(X_T)\mid Y_t\right],\]
where $r$ is instantaneous interest rate.
\section{Functional It$\hat{\text{o}}$ calculus}
A path $f: \Lambda\rightarrow\mathbb R$ is said to have a finite quadratic variation along a subdivision $\pi_n=(0=t_0^n<...<t_n^{k(n)}=T)$ if the measures
\[\xi^n=\sum_{i=0}^{k(n)-1}\left(f(t_{i+1}^n)-f(t_i^n)\right)^2\delta_{t_i^n}\]
where $\delta_t$ is the Dirac measure at $t$, converges weakly to a Radon measure $\xi$ on $[0,T]$ such that
\[[f]_t=\xi([0,t])=[f]_t^c+\sum_{0<s\leq t}\left(\triangle f(s)\right)^2\]
where $[f]^c$ is the continuous part of $[f]$. $[f]$ is called the quadratic variation of $f$ along the sequence $\pi_n$.
 \begin{thm}\label{mart1}
Let $x_t$ be a semimartingale. Let $f(x)$ be a function o the $\mathcal C^{2}$-class. Then, for any $t\in[0,T]$, $f(X_t)$ is again a semimartingale and the following holds:
\begin{eqnarray*} f(X_t)&=&f(X_0)+\int_0^t{\mathcal D}f(X_s)\,\mathrm{d}s+\int_0^t\nabla_{\omega}f(X_s)\,\mathrm{d}x_s+
\frac{1}{2}\int_0^t\nabla_{\omega}^2f(X_s)\,\mathrm{d}[x]_s^c\\
&&+\int_0^t\int_{\mathbb R}\left\{f(z,X_{s-}+\gamma)-f(z,X_{s-})-\gamma\nabla_{\omega}f(z,X_{s-})\right\}\nu(\mathrm{d}z)\,\mathrm{d}s\\
&&+\int_0^t\int_{\mathbb R}\{f(z,X_{s-}+\gamma)-f(z,X_{s-})\}\widetilde{N}(\mathrm{d}s,\mathrm{d}z)
\end{eqnarray*}
where we have suppressed the arguments of $\gamma$.
\end{thm}
The following result is the Feynmann-Kac formula for functionals.
\begin{thm}\label{mart}
Let $x$ follow
\begin{equation}
\mathrm{d}x_t=b(X_t)\,\mathrm{d}t+\sigma(X_t)\,\mathrm{d}w_t+\int_{\mathbb R}\gamma(z,X_{s-})\widetilde{N}(\mathrm{d}t,\mathrm{d}z).
\end{equation}
For suitably integrable $g:\Omega\rightarrow\mathbb R$, $0\leq t\leq{T}$ and $r:\Omega\rightarrow\mathbb R$, we define $f:\Omega\rightarrow\mathbb R$ by
\begin{equation}
f(Y_t)\equiv\mathbb E\left[e^{-\int_t^Tr(Z_u)\,\mathrm{d}u}g(Z_T)\mid Z_t=Y_t\right],
\end{equation}
where $Z_T(u)=Y_t(u)$ for $u\in[0,t]$ and
\begin{equation}
\mathrm{d}Z_u=b(Z_u)\,\mathrm{d}u+\sigma(Z_u)\,\mathrm{d}w_u+\int_{\mathbb R}\gamma(z,Z_{u-})\widetilde{N}(\mathrm{d}u,\mathrm{d}z)
\end{equation}
for $u\in[t,T]$. Then, if $f$ is smooth, it satisfies
\begin{eqnarray*}
&&{\mathcal D}f(X_t)+b(X_t)\nabla_{\omega}f(X_t)-r(X_t)f(X_t)+\frac{1}{2}\sigma^2(X_t)\nabla_{\omega}^2f(X_t)\\
&&+\int_{\mathbb R}\left(f(z,X_t+\gamma)-f(z,X_t)-\gamma\nabla_{\omega}f(z,X_t)\right)\nu(\mathrm{d}z)=0.
\end{eqnarray*}
\end{thm}
\begin{proof}
Set \begin{equation}\label{apple}h(Y_t)\equiv\mathbb E\left[e^{-\int_0^Tr(Z_u)du}g(Z_T)\mid Z_t=Y_t\right]=e^{-\int_0^tr(X_u)du}f(Y_t).\end{equation}
An application of the Functional It$\hat{\text{o}}$ formula yields
\begin{eqnarray*}
\mathrm{d}h(Y_t)&=&e^{-\int_0^tr(X_u)\,\mathrm{d}u}\left\{\left(\nabla_{\omega}f(Y_t)\,\mathrm{d}x_t+({\mathcal D}f(Y_t)-r(Y_t)f(Y_t))\,\mathrm{d}t+\frac{1}{2}\nabla_{\omega}^2f(Y_t)\,\mathrm{d}[x]_t^c\right.\right.\\
&&\left.\left.+\int_{\mathbb R}(f(z,Y_{t-}+\gamma)-f(z,Y_{t-})-\gamma\nabla_{\omega}f(z,Y_{t-}))\nu(\mathrm{d}z)\right)\,\mathrm{d}t\right.\\
&&+\left.\int_{\mathbb R}\left(f(z,Y_{t-}+\gamma)-f(z,Y_{t-})\right)\widetilde{N}(\mathrm{d}t,\mathrm{d}z)\right\}\\
&=&e^{-\int_0^tr(X_u)\,\mathrm{d}u}\left\{\left({\mathcal D}f(Y_t)+b(Y_t)\nabla_{\omega}f(Y_t)-r(Y_t)f(Y_t)+
\frac{1}{2}\sigma^2(Y_t)\nabla_{\omega}^2f(Y_t)\right.\right.\\
&&\left.\left.+\int_{\mathbb R}(f(z,Y_{t-}+\gamma)-f(z,Y_{t-})-\gamma\nabla_{\omega}f(z,Y_{t-}))\nu(\mathrm{d}z)\right)\,\mathrm{d}t\right.\\
&&\left.+\sigma(Y_t)\nabla_{\omega}f(Y_t)\,\mathrm{d}w_t +\int_{\mathbb R}\left(f(z,Y_t+\gamma)-f(z,Y_{t-})\right)\widetilde{N}(\mathrm{d}t,\mathrm{d}z)\right\}.
\end{eqnarray*}
Since $h$ is a martingale from (\ref{apple}), the drift term is zero, and hence the result.
\end{proof}
\begin{thm}
Assume that the assumptions in Theorem \ref{mart} holds, with $r=0$ and $f(Y_t)\equiv\mathbb E[g(Z_T)\mid Z_t=Y_t]$. Then
\begin{eqnarray*}
 g(X_T)&=&\mathbb E[g(X_T)\mid X_0]+\int_0^T\sigma(X_t)\nabla_{\omega}f(X_t)\,\mathrm{d}w_t\\
 &&+\int_0^T\int_{\mathbb R}\left(f(z,X_t+\gamma)-f(z,X_{t-})\right)\widetilde{N}(\mathrm{d}t,\mathrm{d}z).
 \end{eqnarray*}
\end{thm}
\begin{proof}
The result follows by applying Theorems \ref{mart1} and \ref{mart}.
\end{proof}

For the definitions of spaces that we need for the latter sections, we use the notation in \cite{Blan2020}. We equip the space $(\Omega,\mathcal F)$ with a probability $\mathbb P$ such that $\omega$ defines a continuous martingale and $J$ a jump-measure such that $\int_0^t\int_{\mathbb R_0^d}\mid z\mid^2J(\mathrm{d}s,\mathrm{d}z)<\infty$ a.s.. with compensator $\mu$ absolutely continuous in time. The filtration $\mathbb F$ is completed by the $\mathbb P$-null sets. \\
Let $\mathcal L_{\mathbb P}^2([\omega],\mu)$ be the Hilbert space of predictable processes $\phi:[0,T]\times\Omega^p\rightarrow\mathbb R^d$ and $psi:[0,T]\times\mathbb R^d\times\Omega^p\rightarrow\mathbb R^d$ such that
\[\parallel(\phi,\psi) \parallel_{\mathcal L_P^2([\omega],\mu)}^2:=\mathbb E\left[\int_0^t\phi^2(s)\,\mathrm{d}[\omega](s)+\int_0^t\int_{\mathbb R_0^d}\psi^2(s,z)\mu(\mathrm{d}z)\,\mathrm{d}s\right]<\infty\]
and $\mathcal I_{\mathbb P}^2([\omega],\mu)$ be the space os square integrable stochastic integrals
\[\mathcal I_{\mathbb P}^2([\omega],\mu)=\left\{Y=\int_0^{\cdot}\phi(s,\omega)\,\mathrm{d}\omega+\int_0^{\cdot}\int_{\mathbb R_0^d}\psi(s,z)\widetilde{N}(\mathrm{d}s,\mathrm{d}z)\big|(\phi,\psi)\in\mathcal L_{\mathbb P}^2([\omega],\nu)\right\},\]
equipped with the norm
\[ \parallel Y\parallel_{\mathcal I_{\mathbb P}^2(\omega,\mu)}=\mathbb E[\mid Y(T)\mid^2].\]
We note that
\[\mathcal L_{\mathbb P}^2([\omega],\mu)=\mathcal L_{\mathbb P}^2([\omega])\oplus\mathcal L_{\mathbb P}^2(\mu) ~~~\text{and}~~
\mathcal I_{\mathbb P}^2([\omega],\mu)=\mathcal I_{\mathbb P}^2([\omega])\oplus\mathcal I_{\mathbb P}^2(\mu)  \]
with
\[ \mathcal L_{\mathbb P}^2([\omega]):=\left\{\phi:[0,T]\times\Omega^p\rightarrow\mathbb R~~\text{predictable}|~~ \| \phi \|_{\mathcal L_{\mathbb P}^2(\omega)}:=\mathbb E\left[\int_0^T\phi^2(t)d[\omega](t)\right]<\infty,\right\} \]
and
\[\mathcal I_{\mathbb P}^2([\omega]):=\left\{Y:[0,T]\times\Omega^p\rightarrow\mathbb R| Y(t)=\int_0^t\phi^2(z)\,\mathrm{d}\omega(t)\right\},\]
equipped with the norm
\[\parallel Y\parallel_{\mathcal I_{\mathbb P}^2([\omega])}:=\mathbb E[\mid Y\mid_T^2].\]

The stochastic integral operator defined as
\begin{align*}
{I}_{\omega,\widetilde{N}}:\mathcal{L}_{\mathbb P}^2([\omega],\mu)&\to\mathcal{I}_{\mathbb P}^2([\omega],\mu)\\
(\phi,\psi)&\mapsto\int_0^\cdot\phi\mathrm{d}\omega(s)+\int_0^{\cdot}\int_{\mathbb R_0^d}\psi(s,z)\widetilde{N}(\mathrm{d}s,\mathrm{d}z).
\end{align*}
The operators $\nabla_{\omega},\nabla_J: \mapsto\mathcal L_{\mathbb P}^2([\omega],\mu)$ admit suitable extensions to $\mathcal I_{\mathbb P}^2([\omega],\mu)$ which verifies
\[\forall\phi, \psi\in\mathcal L_{\mathbb P}^2([\omega],\mu)~~~\nabla\omega\left(\int_0^{\cdot}\phi \,\mathrm{d}\omega\right)=\phi~~~\text{and}~~~
\nabla_J\left(\int_0^{\cdot}\int_{\mathbb R_0^d}\psi \widetilde{N}(\mathrm{d}s,\mathrm{d}z)\right)=\psi,\]
that is, $\nabla_{\omega}$ and $\nabla_J$ are inverses of stochastic integrals.
\begin{thm}
The operators $\nabla_{\omega},\nabla_J:I_{\omega,\widetilde{N}}(S)\to\mathcal{L}_{\mathbb P}^2([\omega],\mu)$ admit a closure in $\mathcal{I}_{\mathbb P}^2([\omega],\mu)$. Its closure is a bijective isometry
\begin{align*}
\nabla_{\omega},\nabla_J:\mathcal{I}_{\mathbb P}^2([\omega],\mu)&\to\mathcal{L}_{\mathbb P}^2([\omega],\mu)\\
F(t,\omega_t,j_t): \int_0^\cdot\phi\mathrm{d}\omega(s)+\int_0^{\cdot}\int_{\mathbb R_0^d}\psi\widetilde{N}(\mathrm{d}s,\mathrm{d}z)&\mapsto(\nabla_{\omega}F,\nabla_JF)=(\phi,\psi).
\end{align*}
In particular, $\nabla_{\omega},\nabla_J$ are the adjoint of the  stochastic integrals
\begin{align*}
I_{\omega,\widetilde{N}}:\mathcal{L}_{\mathbb P}^2([\omega],\mu)&\to\mathcal{I}_{\mathbb P}^2([\omega,\mu),\\
(\phi,\psi)&\mapsto\int_0^\cdot\phi\mathrm{d}\omega(s)+\int_0^{\cdot}\int_{\mathbb R_0^d}\psi\widetilde{N}(\mathrm{d}s,\mathrm{d}z)
\end{align*}
in the following sense that for all $\phi,\psi\in\mathcal{L}_{\mathbb P}^2([\omega],\mu)$ and for all $Y\in\mathcal{I}_{\mathbb P}^2([\omega],\mu)$
\begin{eqnarray}\label{eq24}
\langle Y,I_{\omega,\widetilde{N}}(\phi,\psi)\rangle_{\mathcal I_{\mathbb P}^2([\omega],\mu)}&=&\mathbb{E}\left[Y(T)\left(\int_0^T\phi(s)\mathrm{d}\omega(s)+\int_0^T\int_{\mathbb R_0^d}\psi(s,z)\widetilde{N}(\mathrm{d}s,\mathrm{d}z)\right)\right]\nonumber\\
&=&\mathbb{E}\left[\int_0^T\nabla_{\omega}Y(s)\phi(s)\mathrm{d}[\omega](s)+\int_0^T\int_{\mathbb R_0^d}\nabla_JY(s,z)\psi(s,z)\mu(\mathrm{d}s,\mathrm{d}z)\right]\nonumber\\
&=:&\langle\nabla_{\omega,J}Y,(\phi,\psi)\rangle_{\mathcal L_{\mathbb P}^2([\omega],\mu)}.
\end{eqnarray}
\end{thm}

\section{Greeks for path-dependent functionals}
In this section, we compute the Greeks for path-dependent options. We consider the following path-dependent volatility model
\begin{align}\label{eq11}
\mathrm{d}x_t=rx_t\mathrm{d}t+\sigma(X_t)\mathrm{d}w_t+\int_{\mathbb R}\gamma(z,X_{t-})\widetilde{N}(\mathrm{d}t,\mathrm{d}z).
\end{align}
The no-arbitrage price at time $t$ be given by
\begin{equation}
f(Y_t)=\exp(-r(T-t))\mathbb E[g(X_T)\mid Y_t],
\end{equation}
where $g:\Omega\to\mathbb{R}$ is a functional. We call $f(y)$ the price function.

It is well-known, see \cite{chan1999}, that, in analogy with the Black-Scholes partial differential equation, in the L\'evy market setting $f$ will satisfy a Partial Differential Integral Equation (PDIE). This is stated in the following result.
\begin{thm}[Pricing pde]
Let $x_t$ be the solution of (\ref{eq11}) and let $r$ be the instantaneous interest rate. Then, if the price of the path-dependent derivative $f$ belongs to $ C^{1,2}$, it satisfies
\begin{eqnarray}
&&{\mathcal D}f(X_t)+rx_t\nabla_{\omega}f(X_t)+\frac{1}{2}\sigma^2(X_t)\nabla_{\omega}^2f(X_t)\nonumber\\
&&+\int_{\mathbb R}\left\{f(z,X_{t-}+\gamma)-f(z,X_{t-})-\gamma\nabla_{\omega}f(z,X_{t-})\right\}\nu(\mathrm{d}z)- rf(X_{t})=0,
\end{eqnarray}
with $f(X_T)=g(X_T)$.
\end{thm}

For the purpose of arithmetic simplicity, we assume that the instantaneous interest rate is zero, that  is, $r=0$. Equation (\ref{eq11}) reduces to
\begin{align}\label{eq73}
\mathrm{d}x_t=\sigma(X_t)\mathrm{d}w_t+\int_{\mathbb R}\gamma(z,X_{t-})\widetilde{N}(\mathrm{d}t,\mathrm{d}z).
\end{align}
The first variation process $\{z_t,~0\leq t\leq T \}$ associated to $\{x_t, ~0\leq t\leq T \}$ given in (\ref{eq73})
is defined by the stochastic differential equation
\begin{align}\label{eq25}
\mathrm{d}z_t=\nabla_{\omega}\sigma(X_t)z_t\mathrm{d}w_t+\int_{\mathbb R}\nabla_{\omega}\gamma(z,X_{t-})z_t\widetilde{N}(\mathrm{d}t,\mathrm{d}z), ~~Z_0=I.
\end{align}
The price of the derivative with maturity $T$ and contract $g:\Omega\rightarrow\mathbb R$ is, therefore, given by the functional $f:\Omega\rightarrow\mathbb R$
 \begin{align}\label{eq13}
f(Y_t)=\mathbb{E}[g(X_T)|Y_t],
\end{align}
for any $Y_t\in\Omega$.

 The pricing pde reduces to
\begin{eqnarray}\label{eq14}
&&{\mathcal D}f(X_t)+\frac{1}{2}\sigma^2(X_t)\nabla_{\omega}^2f(X_t)+\int_{\mathbb R}\left\{f(z,X_{t-}+\gamma)-f(z,X_{t-})-\gamma\nabla_{\omega}f(z,X_{t-})\right\}\nu(\mathrm{d}z)=0.\nonumber\\
\end{eqnarray}
We want to apply the functional It{\^o} formula to $\nabla_{\omega}f(X_t)z_t$. First, we apply $\nabla_\omega$ to both the pricing pde and
functional It{\^o} formula. Applying $\nabla_\omega$ to the pricing pde \eqref{eq14} yields
\begin{align}\label{eq_pde}
  \nabla_\omega(\mathcal{D}f)&(X_t)+\sigma(X_t)\nabla_\omega\sigma(X_t)\nabla_\omega^2f(X_t)+\frac{1}{2}\sigma^2(X_t)\nabla_\omega^3f(X_t)
  +\int_{\mathbb{R}}\left\{f(z,X_{t-}+\gamma)-f(z,X_{t-})\right\}\nu(\mathrm{d}z)\nonumber\\
  &+\int_{\mathbb{R}}\left\{f(z,X_{t-}+\gamma)-f(z,X_{t-})-\gamma\nabla_\omega f(z,X_{t-})\right\}\nu(\mathrm{d}z)=0.\nonumber\\
\end{align}
Applying $\nabla_\omega$ to the functional It{\^o} formula gives
\begin{align*}
  \nabla_\omega f(X_t)&=\nabla_\omega f(X_0)+\int_{0}^{t}\nabla_\omega(\mathcal{D}f)(X_s)\,\mathrm{d}s+\int_{0}^{t}\nabla_\omega^2f(X_s)\,\mathrm{d}x_s+
  \frac{1}{2}\int_{0}^{t}\nabla_\omega^3f(X_s)\,\mathrm{d}[x]_s^c\\
  &+\int_{0}^{t}\int_{\mathbb{R}}\left\{f(z,X_{s-}+\gamma)-f(z,X_{s-})-\gamma\nabla_\omega f(z,X_{s-})\right\}\nu(\mathrm{d}z)\,\mathrm{d}s\nonumber\\
  &+\int_0^t\int_{\mathbb R}\left\{f(z,X_{s-}+\gamma)-f(z,X_{s-})\right\}\widetilde{N}(\mathrm{d}s,\mathrm{d}z),
\end{align*}
which, in differential form, takes the following form
\begin{align}\label{eq74}
 \mathrm{d}(\nabla_\omega f(X_t))&=\nabla_\omega(\mathcal{D}f)(X_t)\,\mathrm{d}t+\nabla_\omega^2f(X_t)\,\mathrm{d}x_t+\frac{1}{2}\nabla_\omega^3f(X_t)\,\mathrm{d}[x]_t^c\nonumber\\
  &+\int_{\mathbb{R}}\left\{f(z,X_{t-}+\gamma)-f(z,X_{t-})-\gamma\nabla_\omega f(z,X_{t-})\right\}\nu(\mathrm{d}z)\,\mathrm{d}t\nonumber\\
  &+\int_{\mathbb{R}}\left\{f(z,X_{t-}+\gamma)-f(z,X_{t-})\right\}\widetilde{N}(\mathrm{d}t,\mathrm{d}z)
\end{align}
%
Substituting (\ref{eq73}) into (\ref{eq74}) gives
\begin{align*}
  \mathrm{d}(\nabla_\omega f(X_t)) &= \nabla_\omega(\mathcal{D}f)(X_t)\,\mathrm{d}t+\sigma(X_t)\nabla_\omega^2f(X_t)\,\mathrm{d}w_t
  +\frac{1}{2}\sigma^2(X_t)\nabla_\omega^3f(X_t)\,\mathrm{d}t\nonumber\\
  &+\int_{\mathbb{R}}\left\{f(z,X_{t-}+\gamma)-f(z,X_{t-})-\gamma\nabla_\omega f(z,X_{t-})\right\}\nu(\mathrm{d}z)\,\mathrm{d}t\nonumber\\
  &+\int_{\mathbb{R}}\left\{f(z,X_{t-}+\gamma)-f(z,X_{t-})\right\}\widetilde{N}(\mathrm{d}t,\mathrm{d}z).
\end{align*}
Now, applying the functional It{\^o} formula to $\nabla_{\omega}f(X_t)z_t$ yields
\begin{align}\label{eq101}
  \mathrm{d}(\nabla_\omega f(X_t)z_t)&= z_t\,\mathrm{d}(\nabla_\omega f(X_t))+\nabla_\omega f(X_t)\,\mathrm{d}z_t+\mathrm{d}(\nabla_\omega f(X_t))\,\mathrm{d}z_t\nonumber\\
  &~~= \nabla_\omega(\mathcal{D}f)(X_t)z_t\,\mathrm{d}t+\sigma(X_t)\nabla_\omega^2f(X_t)z_t\,\mathrm{d}w_t+\frac{1}{2}\sigma^2(X_t)\nabla_\omega^3f(X_t)z_t\,\mathrm{d}t\nonumber\\
  &~~+\int_{\mathbb{R}}\left\{f(z,X_{t-}+\gamma)-f(z,X_{t-})-\gamma\nabla_\omega f(z,X_{t-})\right\}z_t\nu(\mathrm{d}z)\,\mathrm{d}t\nonumber\\
  &~~+\int_{\mathbb{R}}\left\{f(z,X_{t-}+\gamma)-f(z,X_{t-})\right\}z_t\widetilde{N}(\mathrm{d}t,\mathrm{d}z)
  +\nabla_\omega\sigma(X_t)\nabla_\omega f(X_t)z_t\,\mathrm{d}w_t\nonumber\\
  &+\sigma(X_t)\nabla_\omega\sigma(X_t)\nabla_\omega^2f(X_t)z_t\,\mathrm{d}t+\int_{\mathbb{R}}\left\{f(z,X_{t-}+\gamma)-f(z,X_{t-})\right\}z_t\nu(\mathrm{d}z)\,\mathrm{d}t\nonumber\\
          & =\left(\nabla_\omega(\mathcal{D}f)(X_t)+\sigma(X_t)\nabla_\omega\sigma(X_t)\nabla_\omega^2f(X_t)+
          \frac{1}{2}\sigma^2(X_t)\nabla_\omega^3f(X_t)\right.\nonumber\\
          &+\left.\int_{\mathbb{R}}\left\{f(z,X_{t-}+\gamma)-f(z,X_{t-})\right\}\nu(\mathrm{d}z)\right.\nonumber\\
            &+\left.\int_{\mathbb{R}}\left\{f(z,X_{t-}+\gamma)-f(z,X_{t-})
            -\gamma\nabla_\omega f(z,X_{t-})\right\}\nu(\mathrm{d}z)\right)z_t\,\mathrm{d}t\nonumber\\
            &+\left(\sigma(X_t)\nabla_\omega^2f(X_t)+\nabla_\omega\sigma(X_t)\nabla_\omega f(X_t)\right)z_t\,\mathrm{d}w_t\nonumber\\
          &+\left.\left\{f(z,X_{t-}+\gamma)-f(z,X_{t-})\right\}\right)z_t\widetilde{N}(\mathrm{d}t,\mathrm{d}z)
\end{align}
We observe from \eqref{eq_pde} that the first term in \eqref{eq101} is zero and that the the sum of the second and the third terms is a local martingale.

We define a local martingale
\begin{align}\label{eq75}
  m_t&:= \int_{0}^{t}\left(\nabla_\omega\sigma(X_s)\nabla_\omega f(X_s)+\sigma(X_s)\nabla_\omega^2f(X_s)\right)z_s\,\mathrm{d}w_s\nonumber\\
  &~~+\int_0^t\int_{\mathbb{R}}\left\{f(z,X_{s-}+\gamma)-f(z,X_{s-})\right\}z_s\widetilde{N}(\mathrm{d}s,\mathrm{d}z)
\end{align}
with $m_0=0$.

Using Definition \ref{Lie21} and \eqref{eq_pde} we obtain
\begin{align}\label{eq22}
\mathrm{d}(\nabla_\omega f(X_t)z_t)=\left(\mathcal D(\nabla_{\omega}f(X_t))-\nabla_{\omega}(\mathcal Df(X_t))\right)z_t\,\mathrm{d}t=-\mathfrak{L}f(X_t)z_t\mathrm{d}t+\mathrm{d}m_t.
\end{align}
As in \cite{jazaerli2017functional}, we assume that the Lie bracket of $f$, $\mathfrak{L}f$ exists, and that for continuous paths $Y_t$,
 $\mathfrak{L}f(Y_t)=0.$ In particular, $f$ is locally weakly path-dependent.
\subsubsection{Delta}
We are now in a position to compute the price sensitivities. We first compute the Delta. The Delta is defined as the derivative of the option price with respect to the initial price. For a price function $f(Y_t)$ at time $t$, its Delta is given by $\nabla_{\omega}f(Y_t)$. The following theorem gives the Delta of a price.
\begin{thm}\label{delta}
Let $x_t$ be the solution of (\ref{eq73}) on the interval $[0,T]$ and $g:\Lambda_T\to\mathbb{R}$ denote some bounded function. Suppose the price function is given by (\ref{eq13}). Then $(\nabla_\omega f(X_t)z_t)_{t\in[0,t]}$ is a local martingale and the following formula for the Delta is valid
\begin{align}\label{eq21}
\nabla_\omega f(X_0)=\mathbb{E}\left[g(X_T)\frac{1}{T}\int_0^T\frac{z_t}{\sigma(X_t)}\,\mathrm{d}w_t|X_0\right].
\end{align}
\end{thm}
\begin{proof}
Assume that $x$ and $m$ martingales. Since $x_t$ is a continuous path we conclude that
\begin{align}\label{eq23}
\nabla_\omega f(X_t)z_t=\nabla_\omega f(X_0)+m_t,
\end{align}
which clearly shows that $(\nabla_\omega f(X_t)z_t)_{t\in[0,T]}$ is a local martingale. Integrating \eqref{eq23} with respect to $t$, we obtain
\begin{align*}
\int_0^T\nabla_\omega f(X_t)z_t\,\mathrm{d}t=\nabla_\omega f(X_0)T+\int_0^Tm_t\,\mathrm{d}t.
\end{align*}
Taking the expectation on both sides we get
\begin{align*}
\mathbb E\left[\int_0^T\nabla_\omega f(X_t)z_t\,\mathrm{d}t\right]=\nabla_\omega f(X_0)T+\mathbb E\left[\int_0^Tm_t\,\mathrm{d}t\right].
\end{align*}
The dominated convergence theorem allow us to interchange the expectation and the integral in the last term, noticing that $\mathbb E[m_t]=m_0=0$, we have
\begin{align*}
\mathbb{E}\left[\int_0^T\nabla_\omega f(X_t)z_t\,\mathrm{d}t\right]=\nabla_\omega f(X_0)T,
\end{align*}
which is equivalent to
\begin{align*}
\nabla_\omega f(X_0)=\mathbb{E}\left[\int_0^T\nabla_\omega f(X_t)\frac{1}{T}z_t\,\mathrm{d}t\right].
\end{align*}
We note that from \eqref{eq73} that
\[ \mathrm{d}[x]_t^c=\sigma^2(X_t)\,\mathrm{d}t+\int_{\mathbb R}\gamma^2(z,X_t)\nu(\mathrm{d}z)\,\mathrm{d}t.\]
Therefore, we have
\begin{align*}
\nabla_\omega f(X_0)&=\mathbb{E}\left[\int_0^T\nabla_\omega f(X_t)\frac{1}{T}\frac{z_t}{\sigma^2(X_t)}\,
\mathrm{d}[x]_t^c-\int_0^T\int_{\mathbb R}\nabla_{\omega}f(X_t)\frac{1}{T}\frac{z_t}{\sigma^2(X_t)}\gamma^2(z,X_{t-})\nu(\mathrm{d}z)\,\mathrm{d}t]\right].
\end{align*}
An application of the integration by parts formula \eqref{eq24} we obtain
\begin{align*}
\nabla_\omega f(X_0)&=
\mathbb{E}\left[f(X_t)\frac{1}{T}\left(\int_0^T\frac{z_t}{\sigma^2(X_t)}\,\mathrm{d}x_t-\int_0^T\int_{\mathbb R}\frac{z_t}{\sigma^2(X_t)}\gamma(z,X_{t-})\widetilde{N}(\mathrm{d}t,\mathrm{d}z)\right)\right].
\end{align*}
Substituting for $dx_t$ and noting that $f(X_t)$ is a martingale, we obtain the result:
\begin{align*}
\nabla_\omega f(X_0)&=\mathbb{E}\left[g(X_T)\frac{1}{T}\int_0^T\frac{z_t}{\sigma(X_t)}\,\mathrm{d}w_t\right].
\end{align*}
\end{proof}

As in \cite{jazaerli2017functional}, we rewrite Eq. (\ref{eq21}), for any $s\in[0,T]$, as
\begin{align}\label{eq61}
\nabla_{\omega}f(Y_s)=\frac{1}{(T-s)z(Y_s)}\mathbb E\left[g(X_T)\int_s^T\frac{z_t}{\sigma(X_t)}\,\mathrm{d}w_t| Y_s\right],
\end{align}
where $z(Y_s)$ is the functional version of the tangent process $z$. Eq. (\ref{eq61}) can be proved by using similar arguments as in the proof of Theorem \ref{delta} with minor modifications.

\subsubsection{Gamma}
The Gamma is defined as the second derivative of the option price with respect to the initial price. If $f(X_t)$ denotes the option at time $t$, its Gamma is given by $\nabla_\omega^2f(X_t)$.
\begin{asp}\label{3}
$\mathcal{D}\sigma=\mathcal{D}(\nabla_\omega\sigma)=0$ in $\Lambda_T$.
\end{asp}
This assumption is satisfied by time-homogeneous local volatility models, see \eqref{eq11}. Then we assume that the tangent process given by \eqref{eq25} satisfies the following derivative assumptions.
\begin{asp}\label{4}
We consider the functional $z$ such that $z(X_t)=z_t$:
\begin{enumerate}
    \item $\mathcal{D}z_t=0$,
    \item $\nabla_\omega z_t=\frac{\nabla_\omega\sigma(X_{t})}{\sigma(X_{t})}z(X_{t})$,
    \item $\nabla_\omega^2z_t=0$.
\end{enumerate}
\end{asp}

Now, the following result gives the Gamma of the price.
\begin{thm}\label{gamma}
Let $x_t$ be the solution of (\ref{eq73}) on the interval $[0,T]$ and $g:\Lambda_T\to\mathbb{R}$ denote some bounded function. Suppose the price function is given by (\ref{eq13}). In addition, suppose that assumptions \ref{3} and \ref{4} hold. Then
\begin{align}\label{eq26}
\nabla_\omega^2f(X_s)=\mathbb{E}\left[g(X_T)\pi_{s,T}^{\Gamma}\big| X_s\right],
\end{align}
where
\begin{align}\label{puyot5}\pi_{s,T}^{\Gamma}=\frac{(\pi_T-\pi_s)^2}{(T-s)z_s^2}-
\frac{\nabla_\omega\sigma(X_s)}{\sigma(X_s)}\frac{(\pi_T-\pi_s)}{(T-s)z_s}-\frac{1}{(T-s)\sigma^2(X_s)}\end{align} with $\pi$ given by $\displaystyle\pi_s=\int_0^s\frac{z_t}{\sigma(X_t)}\mathrm{d}W_t$.
\end{thm}
\begin{proof}
Suppose that thee exists functionals $z$ and $\pi$ such that $z(X_t)=z_t$ and $\pi(X_t)=\pi_t$ a.s. Then it follows that \[\nabla_\omega\pi(Y_t)=\frac{z(Y_t}{\sigma^2(Y_t)}, ~~~\nabla_{\omega}^2\pi(Y_t)=0,~~~\text{and}~~~ \mathcal{D}\pi(Y_t)=0.\]
 From \eqref{eq61}, we have
\begin{align}\label{eq27}
(T-s)z(Y_s)\nabla_\omega f(Y_s)+f(Y_s)\pi(Y_s)=\mathbb{E}\left[g(X_T)\pi(X_T)\big|Y_s\right].
\end{align}
Let $\tilde{g}(Y_T)=g(Y_T)\pi(Y_T)$ and $\tilde{f}(Y_s)=\mathbb{E}[\tilde{g}(X_T)\big|Y_s]$. Therefore
\begin{align}\label{eq28}
\tilde{f}(Y_s)=(T-s)z(Y_s)\nabla_\omega f(Y_s)+f(Y_s)\pi(Y_s).
\end{align}
To apply similar arguments as in the proof of Theorem \ref{delta}, we need to prove that $\mathcal L \widetilde{f}=0$. Taking the vertical and horizontal derivatives of \eqref{eq28}, respectively, gives
\begin{align}\label{eq29}
\nabla_\omega\tilde{f}(Y_s)&=(T-s)\frac{\nabla_\omega\sigma(Y_s)}{\sigma(Y_s)}z(Y_s)\nabla_\omega f(Y_s)+(T-s)z(Y_s)\nabla_{\omega}^2f(Y_s)+\nabla_\omega f(Y_s)\pi(Y_s)\nonumber\\
&\quad+f(Y_s)\frac{z(Y_s)}{\sigma^2(Y_s)}
\end{align}
and
\begin{align}\label{eq30}
\mathcal{D}\tilde{f}(Y_s)&=-z(Y_s)\nabla_\omega f(Y_s)+(T-s)z(Y_s){\mathcal D}\nabla_{\omega}f(Y_s)+\mathcal D f(Y_s)\pi(Y_s).
\end{align}
Next, we compute the mixed derivatives, that is, we take the horizontal derivative of \eqref{eq29} and the vertical derivative of \eqref{eq30}, respectively,
\begin{align}\label{eq31}
\mathcal{D}(\nabla_\omega\tilde{f})(Y_s)&=-\frac{\nabla_\omega\sigma(Y_s)}{\sigma(Y_s)}z(Y_s)\nabla_{\omega}f(Y_s)+(T-s)\frac{\mathcal D\nabla\omega\sigma(Y_s)}{\sigma(Y_s)}z(Y_s)\nabla\omega f(Y_s)\nonumber\\
&\quad -(T-s)\frac{\nabla_\omega\sigma(Y_s)}{\sigma^2(Y_s)}\mathcal D \sigma(Y_s)z(Y_s)\nabla_{\omega}f(Y_s)+(T-s)\frac{\nabla_\omega\sigma(Y_s)}{\sigma(Y_s)}z(Y_s)\mathcal D \nabla_\omega f(Y_s)\nonumber\\
&\quad -z(Y_s)\nabla_{\omega}^2f(Y_s)+(T-s)z(Y_s)\mathcal D\nabla_{\omega}^2f(Y_s)+\mathcal D\nabla_{\omega}f(Y_s)\pi(Y_s)+\mathcal D f(Y_s)\frac{z(Y_s)}{\sigma^2(Y_s)}\nonumber\\
&\quad-2f(Y_s)\frac{z(Y_s)}{\sigma^3(Y_s)}\mathcal D \sigma(Y_s)
\end{align}
and
\begin{align}\label{eq32}
\nabla_\omega(\mathcal{D}\tilde{f})(Y_s)&=-\frac{\nabla_\omega\sigma(Y_s)}{\sigma(Y_s)}z(Y_s)\nabla_\omega f(Y_s)-z(Y_s)\nabla_{\omega}^2f(Y_s)+(T-s)\frac{\nabla_\omega\sigma(Y_s)}{\sigma(Y_s)}z(Y_s)\mathcal D\nabla_\omega f(Y_s)\nonumber\\
&\quad+(T-s)z(Y_s)\nabla_\omega\mathcal D\nabla_\omega f(Y_s)+\nabla_\omega\mathcal D f(Y_s)\pi(Y_s)+\mathcal D f(Y_s)\frac{z(Y_s)}{\sigma^2(Y_s)}.
\end{align}
From \eqref{eq31} and \eqref{eq32}, we have
\[ \mathcal L\widetilde{f}(Y_s)=\nabla_\omega(\mathcal{D}\tilde{f})(Y_s)-\mathcal{D}(\nabla_\omega\tilde{f})(Y_s)=0. \]
Hence, by Theorem \ref{delta}, $\nabla_\omega\widetilde{f}(X_t)z_t$ is a local martingale.
Therefore,
\begin{align}\label{eq33}
(T-s)z_s\nabla_\omega\widetilde{f}(X_s)+\widetilde{f}(X_s)\int_0^s\frac{z_t}{\sigma(X_t)}\,\mathrm{d}w_t=\mathbb E\left[\widetilde{g}(X_T)\int_0^T\frac{z_t}{\sigma(X_t)}\,\mathrm{d}w_t\big| X_s\right].
\end{align}
From \eqref{eq29}, we have
\begin{align}\label{eq62}
\nabla_\omega\tilde{f}(X_s)&=(T-s)\frac{\nabla_\omega\sigma(X_s)}{\sigma(X_s)}z(X_s)\nabla_\omega f(X_s)+(T-s)z(X_s)\nabla_{\omega}^2f(X_s)\nonumber\\
&\quad+\nabla_\omega f(X_s)\int_0^s\frac{z_t}{\sigma(X_t)}\,\mathrm{d}w_t+f(X_s)\frac{z(X_s)}{\sigma^2(X_s)}.
\end{align}
Using the definition of $\tilde{g}(X_t)$, $f(X_t)$ and $\nabla_\omega f(X_s)$, the desired result follows from Eqs. (\ref{eq33}) and (\ref{eq62}).
\end{proof}
 \subsubsection{Vega}
 In this section, we provide an expression of the the derivative of the expectation $f(X_t)$  with respect to the volatility $v_0=\sigma^2$. 
 We consider the SDE
 \begin{equation}\label{volatility_model_vega}
   \mathrm{d}x_t=\sqrt{v_0}\,\mathrm{d}w_t+\int_{\mathbb{R}}\gamma\widetilde{N}(\mathrm{d}t,\mathrm{d}z),~~v_0=\sigma^2.
 \end{equation}
 By the functional p.d.e
 \begin{align}\label{eqa}
   \mathcal{D}f(X_t) & +\frac{1}{2}v_0\nabla_\omega^2f(X_t)+\int_{\mathbb{R}}\left\lbrace f(z,X_{t-}+\gamma)-f(z,X_{t-})-\gamma\nabla_\omega f(z,X_{t-})\right\rbrace v(\mathrm{d}z)=0.
 \end{align}
 Calculating the expression for Vega is not as straightforward as in the computation of the Delta and Gamma. Here, instead of computing the derivative of the price with respect to the volatility, we define the perturbed process with respect to the property under investigation, in this case $v_0=\sigma^2$. Let $u$ be a direction function for the diffusion such that for every $\varepsilon\in[-1,1]$, $u$ and $v_0+\varepsilon u$ are continuously differentiable with bounded first derivatives in the space direction. The functions $v_0$ and $u$ are assumed to satisfy the following condition:
 \begin{align}\label{puyot4}\exists\eta>0~~~\xi^*(v_0+\varepsilon u)^*(x)(v_0+\varepsilon u)\xi\geq\eta\parallel \xi\parallel^2~~~\forall~~\xi,x\in\mathbb R.\end{align}
 We then define the volatility-perturbed process $x_t^{\varepsilon}$ as a solution of the following perturbed stochastic differential equation
 \begin{align}\label{puyot}
 \mathrm{d}x_t^\varepsilon=\sqrt{v_0+\varepsilon u}\,\mathrm{d}w_t+\int_{\mathbb{R}}\gamma\widetilde{N}(\mathrm{d}t,\mathrm{d}z),~~x_0^\varepsilon=x_0. \end{align}
 The functional It{\^o} formula becomes
 \begin{align}\label{eqb}
   g(X_T^\varepsilon) &=f(X_T^\varepsilon)=f(X_0^\varepsilon)+\int_{0}^{T}\mathcal{D}f(X_t^\varepsilon)\,\mathrm{d}t+\int_{0}^{T}\nabla_\omega f(X_t^\varepsilon)\,\mathrm{d}[x^\varepsilon]_t^c+\frac{1}{2}\int_{0}^{T}(v_0+\varepsilon u)\nabla_\omega^2f(X_t^\epsilon)\,\mathrm{d}t \nonumber\\
   &\quad+\int_{0}^{T}\int_{\mathbb{R}}\left\lbrace f(z,X_{t-}^\varepsilon+\gamma)-f(z,X_{t-}^\varepsilon)-\gamma\nabla_\omega f(z,X_{t-}^\varepsilon)\right\rbrace v(dz)\,\mathrm{d}t\nonumber\\
   &\quad+\int_{0}^{T}\int_{\mathbb{R}}\left\lbrace f(z,X_{t-}^\varepsilon+\gamma)-f(z,X_{t-}^\varepsilon)\right\rbrace\widetilde{N}(\mathrm{d}t,\mathrm{d}z)\nonumber\\
   &=f(X_0^{\varepsilon})+\int_{0}^{T}\nabla_\omega f(X_t^\varepsilon)\,\mathrm{d}[x^\varepsilon]_t^c+\frac{\varepsilon}{2}\int_{0}^{T}u\nabla_\omega^2f(X_t^\varepsilon)\,\mathrm{d}t\nonumber\\
   &\quad+\int_{0}^{T}\int_{\mathbb{R}}\left\lbrace f(z,X_{t-}^\varepsilon+\gamma)-f(z,X_{t-}^\varepsilon)\right\rbrace\widetilde{N}(\mathrm{d}t,\mathrm{d}z)\nonumber\\
 \end{align}
 where the second equality follows by \eqref{eqa}.

 We can relate to the the perturbed process (\ref{puyot}) the perturbed price $f(X_t^{\varepsilon})$ defined by
 \begin{equation}\label{puyot1}
 f(X_t^{\varepsilon})=\mathbb E\left[g(X_T^{\varepsilon})\mid X_t^{\varepsilon}\right].
 \end{equation}
 The Vega is defined, using the Fr{\'e}chet derivative, as
 \begin{equation}\label{eqc}
  \text{Vega}:= \left\langle\nabla_{u}f,u\right\rangle=\lim_{\varepsilon\to0}\frac{f(X_t^{\varepsilon})-f(X_t)} {\varepsilon}.
 \end{equation}
Under similar assumptions to the case of the gamma, the following result gives the expression for the Vega.
\begin{thm}\label{vega}
Assume that condition (\ref{puyot4}) and the hypotheses of theorem \ref{gamma} hold. Then the Vega satisfies
\begin{align}
\left\langle\nabla_{u}f,u\right\rangle=\mathbb{E}\left[g(X_T)\frac{1}{2}\int_0^Tu(X_t)\pi_{t,T}^{\Gamma}\,\mathrm{d}t\right],
\end{align}
where $\pi_{t,T}^{\Gamma}$ is given by (\ref{puyot5}).
\end{thm}
\begin{proof}
Under small perturbations, $\varepsilon \to 0$, the dominant term in the perturbed price $f(X_t^{\varepsilon})$ is:
\begin{align*}
    f(X_t^{\varepsilon}) = f(X_t) + \frac{\varepsilon}{2} \int_0^T u(X_t) \nabla_\omega^2f(X_t) \, \mathrm{d}t.
\end{align*}
The Vega as a Fr{\'e}chet derivative is defined as:
%
%
\begin{align}\label{abv}
\left\langle\nabla_{u}f,u\right\rangle&=\lim_{\varepsilon\rightarrow0}\frac{\mathbb E[g(X_T^{\varepsilon})\mid X_t^{\varepsilon}]-\mathbb E[g(X_T)\mid X_t]}{\varepsilon}.
\end{align}
It follows from the expression of $f(X_t^{\varepsilon})$ that
\[ \langle \nabla_uf,u\rangle = \frac{1}{2} \int_0^T u(X_t) \nabla_\omega^2f(X_t) \, \mathrm{d}t. \]
Assume that the conditions of Theorem \ref{gamma} hold, and noting that
\[\nabla_\omega^2f(X_t)=\mathbb{E}\left[g(X_T)\pi_{t,T}^{\Gamma}\big| X_t\right], \]
we have
\begin{align}
    \langle \nabla_uf,u\rangle = \frac{1}{2} \mathbb{E} \bigg[ g(X_T) \int_0^T u(X_t) \pi_{t,T}^\Gamma \, \mathrm{d}t \bigg].
\end{align}
\end{proof}

\section{Applications}
This section deals with the applications of the previously derived formulas. In particular the section deals with the estimation of Delta, Gamma and Vega of a jump diffusion model  in which the interest rate is assumed to be zero. We focus on the digital option. However, the application could be extended to other discontinuous payoff functions.
We consider the stochastic differential equation
\[\mathrm{d}x_t=\sigma x_t\,\mathrm{d}w_t+\int_{\mathbb R}z\widetilde{N}(\mathrm{d}t,\mathrm{d}z).\]
A direct application of Theorem \ref{delta} yields
\[ \nabla_\omega f(X_0)=\mathbb E\left[g(X_T)\frac{1}{T}\int_0^T\frac{z_t}{\sigma x_t}\,\mathrm{d}w_t\right], \]
where $z_t=\frac{x_t}{x_0}$. Therefore, for a digital option of the type
\[g(X_T)=1_{K<X_T}\]
where $K$ is the exercise price, the Delta is given by
\[\nabla_\omega f(X_0)=\mathbb E\left[1_{K<X_T}\frac{w_T}{x\sigma T}\right]. \]
For Gamma, the application of Theorem \ref{gamma} yields
\[\nabla_\omega^2 f(X_0)=\mathbb E\left[1_{K<X_T}\frac{1}{\sigma x^2 T}\left(\frac{w_T^2}{\sigma T}-w_T-\frac{1}{\sigma}\right)\right].\]
For Vega, the application of Theorem \ref{vega} yields
\[\nabla_\varepsilon f(X_0)=\mathbb E\left[1_{K<X_T}\frac{1}{2}\sigma\left(\frac{w_T^2}{\sigma T}-w_T-\frac{1}{\sigma}\right)\right].\]


\section{conclusion}
Modeling option prices through partial integro-differential equations provides a powerful framework for employing efficient numerical techniques in pricing single-asset options with jumps. This connection has been widely explored by researchers, leading to the development of various numerical approaches for valuing options in jump-diffusion models.

This paper investigates pathwise functional It{\^o} calculus for non-anticipative functionals, establishing a pathwise change of variable formula for functionals of continuous paths. By utilising quadratic variation over a sequence of partitions, we derive a functional pricing partial differential differential equation for path-dependent claims. For the computation of Greeks, we employ the Functional It{\^o} formula together with the Lie bracket property. Explicit analytic expressions for Delta, Gamma, and Vega are derived, with results showing that these sensitivities can be represented as expectations of a product of the payoff functional and a weight function. Under the B-S model, it is observed that Delta is not a martingale, a consequence of modeling stock prices through their tangent process and incorporating path-dependence in derivative pricing.

Future work should explore extensions to American options, stochastic volatility models, and multi-asset options, broadening the applicability of Functional It{\^o} calculus in mathematical finance.

%

 \end{document}